   \documentclass[twoside,reqno,11pt]{fcaa-var} %
\usepackage{graphicx}
\usepackage{epsfig}

\usepackage{amsthm}
\usepackage{amsmath}
\usepackage{latexsym}
\usepackage{amsfonts}
\usepackage{amssymb}

\newtheoremstyle{theorem}% name
  {15pt}          % space above
  {15pt}  % space below
  {\sl}  % bofy font
  {\parindent}
  {\sc}  % thm head font
  {. }   % punctuation after thm head
  { }    % space after thm head: `` ``=normal \newline=linebreak
  {}     % thm head specification
\theoremstyle{theorem}
\newtheorem{lemma}{Lemma}[section]
\newtheorem{theorem}{Theorem}[section]

\newtheoremstyle{defi}% name
  {15pt}          % space above
  {15pt}  % space below
  {\rm}  % bofy font
  {\parindent}     % ident - empty=no indent,  \parindent= paragraph indent
  {\sc}  % thm head font
  {. }    % punctuation after thm head
  { }    % space after thm head: `` ``=normal \newline=linebreak
  {}     % thm head specification
\theoremstyle{defi}
\newtheorem{definition}{Definition}[section]
\newtheorem{remark}{Remark}[section]

 \def\proofname{\indent\rm P\ r\ o\ o\ f.}
 
 \def\qed{\hfill$\Box$} % instead of \square
   \usepackage{hyperref}
   \usepackage{amsmath}
   \usepackage{mathtools}
   \DeclarePairedDelimiter\ceil{\lceil}{\rceil}
   
  \def\theequation{\arabic{section}.\arabic{equation}}

 \def\themycopyrightfootnote{\vspace*{3pt}
 \copyright \, Year\,  Diogenes  Co., Sofia
   \par  \noindent pp. xxx--xxx, DOI: ......................
   \hfill  \vspace*{-36pt}
  }
\title[Dzherbashian-Nersesian Fractional Operator\dots]{Inverse Problems for diffusion equation with Fractional Dzherbashian-Nersesian Operator\\ [3pt]}
\author{\normalsize Anwar Ahmad$^1$, Muhammad Ali$^2$, Salman A. Malik $^3$}

\begin{document}

 \vbox to 1.5cm { \vfill }

%%% to make empty space that will be replaced
%%% by Editor with the journal's and publishers logos %%%%%%%%

\bigskip
\medskip

%%%% Abstract %%%%%%%%%%%%%%%%%%%%%%%%%
 \begin{abstract}
Fractional Dzherbashian-Nersesian operator is considered and three famous fractional order derivatives namely Riemann-Liouville, Caputo and Hilfer derivatives are shown to be special cases of the earlier one. The expression for Laplace transform of fractional Dzherbashian-Nersesian operator is constructed.
Inverse problems of recovering space dependent and time dependent source terms of a time fractional diffusion equation with involution and involving fractional Dzherbashian-Nersesian operator are considered. The results on existence and uniqueness for the solutions of inverse problems are established. The results obtained here generalize several known results.
 \medskip

%%% these are only examples, pls. out the true MSC for the topics of your paper %%%%

 \smallskip

{\it Key Words and Phrases}: Fractional differential equations, inverse problems, Mittag-Leffler function
%%% these are only examples .. %%%

 \end{abstract}

 \maketitle

%%%%%%% end make title %%%%%%%%%%%%%%%%%%%%%%%%%%%%%%%%%%
 \vspace*{-20pt}

%%%%%%%% begin papers' body %%%%%%%%%%%%%%%%%%%%%%%%%%%%%

%%%%%%%%%%%%%%%%%%%%%%%%%%% Section 1 %%%%%%%%%%%%%

 \section{Introduction}\label{sec:1}

\setcounter{section}{1}
\setcounter{equation}{0}\setcounter{theorem}{0}
The study of fractional differential equations (FDEs) has been largely motivated by their vast and fascinating applications in chemistry \cite{chem1}, physics \cite{phy1}, engineering \cite{engg1} and many other areas of sciences that have evolved in last few decades. Nonlocality of fractional operators is the reason behind the success of FDEs in modeling natural phenomena. This property makes these operators suitable to describe the long memory or nonlocal effects characterizing most physical phenomena. For in depth study of FDEs, we refer the reader to \cite{samko-book}, \cite{Kilb-Sri}.

There are several definitions of fractional order integrals and derivatives that are available in the literature. The characteristics of a fractional order derivative are described by many authors. We refer, for example, to papers \cite{LuckoNTH}, \cite{Ort}. Moreover, Tarasov proposed that fractional order derivatives of non-integer order cannot satisfy the Leibniz rule (see \cite{taraLR},\cite{taraIJACM}, \cite{taraCNS}). The objective of this article is twofold; first to consider and explore some properties of the forgotten fractional derivative proposed by Dzherbashian-Nersesian in \cite{Dzhr-Ner}, which did not get the considerable attention of scientific community and is a generalization of the Rieman-Liouville, Caputo and Hilfer derivatives. The translation of Russian version of \cite{Dzhr-Ner} has been published in FCAA \cite{Dzhertrans}. Secondly, we consider two inverse source problems for a diffusion equation involving Dzherbashian-Nersesian derivative (see Section  \ref{StatementProblem}).

Solving an equation in a specified region subject to certain given data is called direct problem. At the same time, determining an unknown input which could either be some coefficients, or a source function in equation, by utilizing output is called an inverse problem. The inverse problem is known as inverse problem of coefficient identification or inverse problem of source identification  in accordance with this unknown input respectively. In the last few years, there has been increasing interest in investigating the inverse problems of time fractional differential equations (see \cite{AliFCAA}, \cite{Ali-MMA},  \cite{Ali-SalmantimeMMA}, \cite{Ali-Salmanipse}).

Over the past several years, differential equations with involutions have attracted the considerable amount of interest from both the theoretical and practical applications perspectives, for instance (see \cite{InvolutionDE1}, \cite{InvolutionDE2}). Moreover, inverse problems involving differential equations with involution have been investigated in recent times, (see \cite{Kiran-Saltiwinvolution}, \cite{Salti-Kiraninvolution}, \cite{Tore-Tapd}). However, study on inverse problems in fractional differential equations with involution is still in its embryonic stages.

The rest of the article is organized as follows. In the next section we state the inverse problems and over-specified conditions. Section \ref{prelim} is devoted to preliminaries and the spectral problem is discussed in Section \ref{SpectralProblem}. In Section \ref{mainresults}, under the specific choices of parameters, some well known fractional derivatives are proved to be particular cases of fractional Dzherbashian-Nersesian operator. Furthermore, Laplace transform of fractional Dzherbashian-Nersesian operator is derived and the main results about existence and uniqueness of solution of the inverse problems are proved. Finally the article is concluded in Section \ref{conclusions}.

\section{Statement of Problems}\label{StatementProblem}
In this paper, we are interested in the study of two inverse source problems pertaining to the following time fractional differential equation involving spatial involution
\begin{equation}\label{prbmDN}
	\mathcal{D}^{\varrho_{m}}_{0+,t}u(x,t)-u_{xx}(x,t)+\varepsilon u_{xx}(\pi -x,t)=F(x,t),\:\:
	(x,t)\in\Omega,
\end{equation}
subject to initial conditions
\begin{equation}\label{initDN}
	\mathcal{D}^{\varrho_{n}}_{0+,t}u(x,t)|_{t=0}=\varphi_{n}(x),\quad x\in (0,\pi), \quad n=0,...,m-1,\quad m\in \mathbb{N},
\end{equation}
and boundary conditions
\begin{equation}\label{bndryDN}
	u(0,t)=0=u(\pi,t),\quad t\in(0,T],
\end{equation}
where $\mathcal{D}^{\varrho_{k}}_{0+,t}$ stands for Dzherbashian-Nersesian fractional operator of order $\varrho_{k}$ such that $0<\varrho_{k}\leq m$, $\varepsilon$ is a real number and $\Omega:=(0,\pi)\times(0,T]$ (see Section \ref{prelim}).

In the first inverse source problem, we reckon the source term $F(x,t)$ depends only on the space variable, i.e., $F(x,t):=f(x)$. We shall determine the source term $f(x)$ and $u(x,t)$ assuming the following over-determination condition for unique solvability  of (\ref{prbmDN})-(\ref{bndryDN})
\begin{eqnarray}\label{spaceoverdetcdn}
	u(x,T)=\psi(x),\quad t<T.
\end{eqnarray}
By a regular solution of space dependent inverse source problem, we mean a pair of functions $\{u(x,t), f(x)\}$ such that $t^{\varrho_{m}}u(.,t) \in C^{2}([0,\pi]),\\ t^{\varrho_{m}}\mathcal{D}^{\varrho_{m}}_{0+,t}u(x,.) \in C([0,T])$ and $f(x) \in C([0,\pi])$.\\
In the second inverse source problem, we consider the source term as $F(x,t):=a(t)f(x,t)$. Whilst $f(x,t)$ is known, we are keen in recovering the time dependent term $a(t)$ and $u(x,t)$.
We propose total energy of the system $E(t)$ as the over-determination condition to have inverse problem (\ref{prbmDN})-(\ref{bndryDN}) uniquely solvable given by
\begin{eqnarray}\label{timeoverdetcdn}
	\int_{0}^{\pi}u(x,t)dx=E(t),\quad t\in(0,T].
\end{eqnarray}
A regular solution of the time dependent inverse source problem is the pair of functions $\{u(x,t), a(t)\}$ such that $t^{\varrho_{m}}u(.,t) \in C^{2}([0,\pi]),\; t^{\varrho_{m}}\mathcal{D}^{\varrho_{m}}_{0+,t}u(x,.) \in C([0,T])$ and $a(t) \in C([0,T])$.
\section{Preliminaries}\label{prelim}
\noindent In this section, we present some elementary definitions and notions for readers' convenience. Some basic results about Mittag-Leffler functions are also presented.
\begin{definition}\cite{samko-book}
	Let us denote by $AC^{n}[a, b]$, where $n\in \mathbb{N}$, the space of functions $g(t)$ which have continuous derivatives up to order $n-1$ on $[a,b]$ with $g^{(n-1)}(t)\in AC^{n}([a, b])$.
\end{definition}
%\begin{definition}
%	We denote by $C[a, b]$ the space of continuous functions $g$ on $\Omega$ with the norm
%	\begin{align*}
%		\|g\|:=\max_{x\in \Omega} |g(x)|
%	\end{align*}
%\end{definition}
%\begin{definition}
%	Chebyshev norm is defined as
%	\begin{align}
%		\|g\|:=\max_{0\leq t \leq T}|g(t)|.
%	\end{align}
%\end{definition}
\begin{definition}\cite{samko-book},\cite{Kilb-Sri}
	Let $g(t)\in L_{1}([a, b])$, the left sided Riemann-Liouville fractional integral $J^{\zeta}_{0_+,t}$ of order $\zeta$ is defined as
\begin{align*}		 J^{\zeta}_{0_+,t}g(t):=\frac{1}{\Gamma(\zeta)}\int_{0}^{t}\frac{g(\tau)}{(t-\tau)^{1-\zeta}}d\tau, \quad \zeta>0.
	\end{align*}
\end{definition}	
\begin{definition}\cite{samko-book},\cite{Kilb-Sri}\label{RL-def}
	Let $g(t)\in L_{1}([a,b])$, the left sided Riemann-Liouville fractional derivative $D^{\zeta}_{0_+,t}$ of order $\zeta$ is defined as
	\begin{align*}
		D^{\zeta}_{0_+,t}g(t):= \frac{d^{n}}{dt^{n}}J^{n-\zeta}_{0_+,t}g(t)=\frac{1}{\Gamma(n-\zeta)}\frac{d^{n}}{dt^{n}}\int^t_0
		\frac{g(\tau)}{(t-\tau)^{1+\zeta-n}}d\tau, \: \:n=\ceil*{\zeta}.
	\end{align*}
\end{definition}
\begin{definition}\cite{Dzhr-Ner}\label{DzherNer-def}
	Dzherbashian-Nersesian fractional operator $\mathcal{D}^{\varrho_{m}}_{0+,t}$ of order $\varrho_{m}$ is defined as
	\begin{align}
		\mathcal{D}^{\varrho_{m}}_{0+,t} g(t):=J^{1-\zeta_{m}}_{0+,t}D^{\zeta_{m-1}}_{0+,t}D^{\zeta_{m-2}}_{0+,t}...D^{\zeta_{1}}_{0+,t}D^{\zeta_{0}}_{0+,t}g(t),\quad m\in \mathbb{N},\quad t>0,\label{DzhrNer-mathematical}
	\end{align}
\end{definition}
\noindent where $\varrho_{m} \in (0,m]$ is given by
\begin{align*}
\varrho_{m}=\sum_{j=0}^{m}\zeta_{j}-1>0,\quad \zeta_{j}\in (0,1].
\end{align*}
It may be noted that $J^{\xi}_{0+,t}$ and $D^{\xi}_{0+,t}$ are the Riemann-Liouville fractional integral and Riemann-Liouville fractional derivative of order $\xi$ respectively.

%related with the sequence
%\begin{align*}
%	\displaystyle \{\zeta_{k}\}^{m}_{0},\quad \zeta_{k}\in (0,1],\quad k=0,1,...,m.
%\end{align*}
%It may be noted that $D^{\zeta}_{0+,t}$ is the Riemann-Liouville fractional derivative of order $\zeta$.\\
Specifically, for $m=1$ in (\ref{DzhrNer-mathematical}), we have
\begin{align*}
	\mathcal{D}^{\varrho_{1}}_{0+,t} g(t):=J^{1-\zeta_{1}}_{0+,t}	 D^{\zeta_{0}}_{0+,t}g(t)
\end{align*}

\begin{definition}\cite{GoreMittagBook}
	The two parameter Mittag-Leffler function is defined as
	\begin{align*}
E_{\beta,\zeta}(z):=\sum_{k=0}^{\infty}\frac{z^{k}}{\Gamma(\beta k+ \zeta)},\quad Re(\beta)>0,\; \zeta,z \in\mathbb{C}.	
	\end{align*}
\end{definition}
For $\zeta=1$, $E_{\beta,\zeta}(z)$ reduces to the classical Mittag-Leffler function, i.e.,
\begin{align*}
	 E_{\beta,1}(z):=E_{\beta}(z)=\sum_{k=0}^{\infty}\frac{z^{k}}{\Gamma(\beta k+ 1)}.	
\end{align*}
\noindent Moreover, the Mittag-Leffler type function is defined as
\begin{align*}
	e_{\beta,\zeta}(t;\lambda):=t^{\zeta-1}E_{\beta,\zeta}(-\lambda t^{\beta}),\quad Re(\beta)>0,\; \zeta \in\mathbb{C},\;t,\;\lambda>0.
\end{align*}
\begin{lemma}\cite{podlubny}\label{podlem}
	If $\beta<2$, $\zeta$ is an arbitrary real number, $\mu$ is such that $\pi\beta/2<\mu<\min\{\pi,\pi\beta\}$, $z\in\mathbb{C}$ such that $|z|\geq0$,
	$\mu\leq |arg(z)|\leq\pi$ and $C_{1}$ is a real constant, then
	\[
	|E_{\beta,\zeta}(z)|\leq \frac{C_{1}}{1+|z|}.
	\]
\end{lemma}

\begin{lemma}\cite{Ali-MMA}
	For $g(t) \in C([0,T])$, the following property holds
	\begin{align*}
		|g(t)*e_{\zeta,\zeta}(t;\lambda)| \leq \frac{C_{1}}{\lambda}\|g\|_{t},\quad \zeta, \;t,\; \lambda>0,
	\end{align*}
	where $C_{1}$ is a positive constant and $\|.\|_{t}$ is Chebyshev norm defined as
	\begin{align*}
	\|g\|_{t}:=\max_{0\leq t \leq T}|g(t)|.
	\end{align*}
\end{lemma}
\begin{lemma}\cite{AliFCAA}\label{AliMittagtypeLemma1}
	The Mittag-Leffler type functions $e_{\zeta,\zeta+1}(t;\lambda)$ have the following property
	\begin{align*}
		 e_{\zeta,\zeta+1}(t;\lambda)=\frac{1}{\lambda}\biggl(1-e_{\zeta,1}(t;\ \lambda)\biggr),\quad t,\; \lambda>0.
	\end{align*}
\end{lemma}
\begin{lemma}\cite{AliFCAA}\label{AliMittagtypeLemma1'}
	The Mittag-Leffler type functions $e_{\zeta,1}(T;\lambda)$, where  $T, \lambda>0$, possess the following property for $0<\zeta<1$
	\begin{align*}
		\frac{1}{1-e_{\zeta,1}(T; \lambda)}\leq C_{2},
	\end{align*}
	where $C_{2}$ is a positive constant.
\end{lemma}	

\section{Spectral Problem}\label{SpectralProblem}
\noindent The spectral problem for (\ref{prbmDN})-(\ref{bndryDN}) is
\begin{eqnarray}\label{eqspecprbm}
	X''(x)-\varepsilon X''(\pi-x)+\lambda X(x)=0, \quad 0<x<\pi,\: |\varepsilon|<1,
\end{eqnarray}
with boundary conditions
\begin{eqnarray}\label{specbn}
	X(0)=0= X(\pi).
\end{eqnarray}
\noindent The spectral problem (\ref{eqspecprbm})-(\ref{specbn}) has the eigenvalues
\begin{align*}
	\lambda_{1}=(1-\varepsilon),\;
	\lambda_{2k+1}=(1-\varepsilon)(2k+1)^{2},\;
	\lambda_{2k}=(1+\varepsilon)4k^{2},
	\; k \in \mathbb{N},
\end{align*}
and corresponding eigenfunctions are
\begin{align*}
	X_{1}(x)=\sqrt{\frac{2}{\pi}}\sin x,\;
	X_{2k+1}(x)=\sqrt{\frac{2}{\pi}}\sin(2k+1)x,\;
	X_{2k}(x)=\sqrt{\frac{2}{\pi}}\sin2kx,
\end{align*}
where $k\in \mathbb{N}$. The set of eigenfunctions $\{X_{k}:k\in \mathbb{N}\}$ form an orthonormal basis in $L^{2}([0,\pi])$ \cite{Tore-Tapd}.

It may be noted that
\begin{align}\label{lamdasestimate}
	\frac{1}{\lambda_{2k}}\leq \frac{1}{(1+\varepsilon)k^{2}},\;\; \frac{1}{\lambda_{2k+1}}\leq \frac{1}{(1-\varepsilon)k^{2}},\; k\in \mathbb{N}.
\end{align}
%Moreover, we notice that $\lambda_{2k+1}\leq (1-\varepsilon)9k^{2}$.
%\section{Auxiliary Results}\label{AuxiliaryResults}
%\noindent In this section,  we present some useful results which play vital role in proving the main results. At the first consideration, we present how under certain conditions, Dzherbashian-Nersesian fractional operator becomes Riemann-Liouville, Caputo and Hilfer fractional derivative.
The following lemmata hold for the set of functions $\{X_{k}(x): k\in \mathbb{N}\}$.
\begin{lemma}\label{estimatelemma1}
	Let $|\varepsilon|<1$, and $g(x)\in C^{2}([0,\pi])$ be such that $g(0)=0=g(\pi)$. Then the following condition holds
	\begin{align*}
		|g_{k}| \leq \frac{1}{k^{2}}\|g''\|_{x}, \quad k \in \mathbb{N}.
	\end{align*}
where $\|.\|_{x}$ is norm in $L^{2}(0,\pi)$ and is defined as
\begin{align*}
	\|g\|_{x}:= \sqrt{\langle g,g \rangle}.
\end{align*}
Here $\langle .,.\rangle$  denotes the inner product defined as $\langle f,g\rangle:=\int_{0}^{\pi}f(x)g(x)dx.$
	\begin{proof}
		Consider $g_{2k+1}=\langle g(x),X_{2k+1}(x)\rangle$.
		\begin{align*}
			g_{2k+1}=\int_{0}^{1}g(x)\sqrt{\frac{2}{\pi}}\sin(2k+1)x dx,
		\end{align*}
		Integration two times by parts yields,
		\begin{align*}
			 g_{2k+1}=\sqrt{\frac{2}{\pi}}\frac{g(0)+g(\pi)}{2k+1}-\frac{1}{(2k+1)^{2}}\langle g''(x),X_{2k+1}(x)\rangle.
		\end{align*}
		On using the boundary conditions $g(0)=0=g(\pi)$ and Cauchy Schwartz inequality, we obtain
		\begin{align*}
			|g_{2k+1}| \leq \frac{1}{k^{2}}\|g''\|_{x}, \quad k \in \mathbb{N}.
		\end{align*}
	\end{proof}
\end{lemma}
\begin{lemma}\label{estimatelemma2}
	Let $|\varepsilon|<1$, and $g(x)\in C^{4}([0,\pi)]$ be such that $g^{(i)}(0)=0=g^{(i)}(\pi)$ for $i=0,2$. Then we have the following condition
	\begin{align*}
		|g_{k}| \leq \frac{1}{k^{4}}\|g^{(iv)}\|_{x}, \quad k \in \mathbb{N}.
	\end{align*}
	\begin{proof}
		The relation can be obtained  on the same lines as in the proof of Lemma \ref{estimatelemma1}.
	\end{proof}	
\end{lemma}
%\begin{lemma}\label{estimatelemma1time}
%	Let $|\varepsilon|<1$, and $g(.,t)\in C^{2}([0,\pi])$ be such that $g(0,t)=0=g(\pi,t)$. Then the following condition holds
%	\begin{align*}
%		|g_{k}(t)| \leq \frac{1}{k^{2}}\Big\|\frac{\partial^{2}g(x,t)}{\partial x^{2}}\Big\|_{1},  \quad k \in \mathbb{N}.
%	\end{align*}
%where $g_{k}(t):=\big \langle g(x,t), X_{k}(x) \big\rangle$ and $\|(.,t)\|_3$ is the norm defined by 
%\begin{align*}
%	\|g(.,t)\|_3:=\sqrt{|g(.)g(t)|^{2}d(.)dt}.
%\end{align*}

%	where $g_{k}(t):=\langle g(x,t), X_{k}(x) \rangle$.	
%	\begin{proof}
%		Consider $g_{2k+1}(t)=\langle g(x,t),X_{2k+1}(x)\rangle$.
%		\begin{align*}
%			g_{2k+1}(t)=\int_{0}^{1}g(x,t)\sqrt{\frac{2}{\pi}}\sin(2k+1)x dx,
%		\end{align*}
%		Integration two times by parts yields,
%		\begin{align*}
%			g_{2k+1}(t)=\sqrt{\frac{2}{\pi}}\frac{g(0,t)+g(\pi,t)}{2k+1}-\frac{1}{(2k+1)^{2}}\langle \frac{\partial^{2}g(x,t)}{\partial x^{2}},X_{2k+1}(x)\rangle.
%		\end{align*}
%		On using the boundary conditions $g(0,t)=0=g(\pi,t)$ and Cauchy-Bunyakovsky-Schwarz inequality, we obtain
%		\begin{align*}
%			|g_{2k+1}(t)| \leq \frac{1}{k^{2}}\Big\|\frac{\partial^{2}g(x,t)}{\partial x^{2}}\Big\|_{x,t}, \quad k \in \mathbb{N}.
%		\end{align*}
%	\end{proof}
%\end{lemma}
\section{Main Results}\label{mainresults}
In this section, we are going to present our main results. At first some interesting facts about fractional Dzherbashian-Nersesian operator are revealed. Next, we establish the existence and uniqueness results for the solution of inverse problems (\ref{prbmDN})-(\ref{spaceoverdetcdn}) and (\ref{prbmDN})-(\ref{bndryDN}) alongside (\ref{timeoverdetcdn}).
\subsection{Some interesting facts and important results} In this subsection, we prove in different cases how under certain fixation of parameters, fractional Dzherbashian-Nersesian operator is generalization to some well known fractional derivatives. Next, we construct the formula for Laplace transform of fractional Dzherbashian-Nersesian operator. Furthermore, we extend the Lemma 15.2 of \cite{samko-book} for fractional Dzherbashian-Nersesian operator.

	Case-I: For $\zeta_{m}=...=\zeta_{1}=1$ and $\zeta_{0}=1+\zeta-m$, where $\zeta_{0}\in (0,1)$, Equation  (\ref{DzhrNer-mathematical}) interpolates the Riemann-Liouville fractional derivative of order $\zeta \in (m-1,m)$, i.e.,
	\begin{align*}
		\mathcal{D}^{\varrho_{m}}_{0+,t} g(t)=\frac{d^{m}}{dt^{m}}J^{m-\zeta}_{0+,t}g(t)=D^{\zeta}_{0_+,t}g(t).
	\end{align*}
	In particular, for $m=1$, $\zeta_{1}=1$ and $0<\zeta_{0}=\zeta<1$,
	\begin{align*}
		\mathcal{D}^{\varrho_{1}}_{0+,t} g(t)=\frac{d}{dt}J^{1-\zeta}_{0+,t}g(t),
	\end{align*}
	which is Riemann-Liouville fractional derivative of order $\zeta \in (0,1)$.
	
    Case-II: For $\zeta_{m-1}=...=\zeta_{0}=1$ and $\zeta_{m}=1+\zeta-m$, where $\zeta_{m}\in (0,1)$, Equation (\ref{DzhrNer-mathematical}) reduces to another well known Caputo fractional derivative of order $\zeta \in (m-1,m)$, i.e.,
	\begin{align*}
		\mathcal{D}^{\varrho_{m}}_{0+,t} g(t)=J^{m-\zeta}_{0+,t}\frac{d^{m}}{dt^{m}}g(t)={}^{c}D^{\zeta}_{0_+,t}g(t).
	\end{align*}
	More specifically for $m=1$, $0<\zeta_{1}=\zeta<1$ and $\zeta_{0}=1$,
	\begin{align*}
		\mathcal{D}^{\varrho_{1}}_{0+,t} g(t)=J^{1-\zeta}_{0+,t}\frac{d}{dt}g(t),
	\end{align*}
	which is Caputo fractional derivative of order $\zeta \in (0,1)$.
	
    Case-III: For $m\geq 2$, $\zeta_{m-1}=...=\zeta_{1}=1$,  $\zeta_{m}=1-\beta(m-\zeta)$, and $\zeta_{0}=1-(m-\zeta)(1-\beta)$, where $0<\zeta_{0}, \zeta_{m}<1$, Equation (\ref{DzhrNer-mathematical}) gives rise to famous Hilfer fractional derivative of order $\zeta \in (m-1,m)$ and type $\beta\in[0,1]$, i.e.,
	\begin{align*}
		\mathcal{D}^{\varrho_{m}}_{0+,t} g(t)= J^{\beta(m-\zeta)}_{0_+,t}\frac{d^{m}}{dt^{m}}J^{(m-\zeta)(1-\beta)}_{0_+,t}g(t),
	\end{align*}
	More particularly, for $\zeta_{0}=1-(1-\zeta)(1-\beta)$ and $\zeta_{1}=1-\beta(1-\zeta)$, where $0<\zeta_{0}, \zeta_{1}<1$, then
	\begin{align*}
		\mathcal{D}^{\varrho_{1}}_{0+,t} g(t)= J^{\beta(1-\zeta)}_{0_+,t}\frac{d}{dt}J^{(1-\zeta)(1-\beta)}_{0_+,t}g(t),
	\end{align*}
	which is Hilfer fractional derivative of order $\zeta \in (0,1)$ and type $\beta\in[0,1]$.
\begin{lemma}
	Laplace transform of Dzherbashian-Nersesian fractional operator of order $0<\varrho_{m}<m$ is given as
	\begin{align}
		 \mathcal{L}\{\mathcal{D}^{\varrho_{m}}_{0+,t}g(t)\}=s^{\varrho_{m}}\mathcal{L}\{g(t)\}-\sum_{k=1}^{m}s^{\varrho_{m}-\varrho_{m-k}-1}\;\mathcal{D}^{\varrho_{m-k}}_{0+,t}g(t)\big|_{t=0}.\label{LaplaceDN}
	\end{align}
\end{lemma}
\begin{proof}
	From Equation (\ref{DzhrNer-mathematical}),
	\begin{align*}
		\mathcal{D}^{\varrho_{m-k}}_{0+,t} \equiv J^{1-\zeta_{m-k}}_{0+,t}D^{\zeta_{m-k-1}}_{0+,t}D^{\zeta_{m-k-2}}_{0+,t}...D^{\zeta_{1}}_{0+,t}D^{\zeta_{0}}_{0+,t},\quad k=1,...,m,
	\end{align*}
	\begin{align*}
		\mathcal{D}^{\varrho_{m+1}}_{0+,t}\equiv J^{1-\zeta_{m+1}}_{0+,t}\mathcal {D}^{\varrho_{m}+1}_{0+,t},\:\mathcal {D}^{\varrho_{0}+1}_{0+,t}\equiv D^{\zeta_{0}}_{0+,t}.
	\end{align*}
	Taking Laplace transform of (\ref{DzhrNer-mathematical})
	\begin{align*}
		\mathcal{L}\{\mathcal {D}^{\varrho_{m}}_{0+,t}g(t)\}&=\mathcal{L}\{J^{1-\zeta_{m}}_{0+,t}D^{\zeta_{m-1}}_{0+,t}D^{\zeta_{m-2}}_{0+,t}...D^{\zeta_{1}}_{0+,t}D^{\zeta_{0}}_{0+,t}g(t)\},\\
		& =s^{\zeta_{m}-1}\mathcal{L}\{D^{\zeta_{m-1}}_{0+,t}D^{\zeta_{m-2}}_{0+,t}...D^{\zeta_{1}}_{0+,t}D^{\zeta_{0}}_{0+,t}g(t)\},\\
		 &=s^{\zeta_{m}-1}(s^{\zeta_{m-1}}\mathcal{L}\{D^{\zeta_{m-2}}_{0+,t}...D^{\zeta_{1}}_{0+,t}
		D^{\zeta_{0}}_{0+,t}g(t)\}\\
		&-\big(D^{\zeta_{m-1}-1}_{0+,t}
		 (D^{\zeta_{m-2}}_{0+,t}...D^{\zeta_{1}}_{0+,t}D^{\zeta_{0}}_{0+,t}g(t))\big)|_{t=0}).
	\end{align*}
	Proceeding in the similar manner, we eventually have
	\begin{align*}
		\mathcal{L}\{\mathcal {D}^{\varrho_{m}}_{0+,t}g(t)\}=s^{\varrho_{m}}\mathcal{L}\{g(t)\}-\sum_{k=1}^{m}s^{\varrho_{m}-\varrho_{m-k}-1}\;\mathcal {D}^{\varrho_{m-k}}_{0+,t}g(t)\big|_{t=0}.
	\end{align*}
\end{proof}
\begin{remark}
	On the substitution of $\zeta_{m}=...=\zeta_{1}=1$ and $\zeta_{0}=1+\zeta-m$, where $\zeta_{0}\in (0,1)$, in (\ref{LaplaceDN}), we have
	\begin{align*}	
		\mathcal{L}\{\mathcal {D}^{\varrho_{m}}_{0+,t}g(t)\}=s^{\zeta}\mathcal{L}\{g(t)\}-\sum_{i=0}^{m-1}s^{m-i-1}\big({D}^{{i}}_{0+,t}J^{m-\zeta}_{0+,t}g(t)\big)_{t=0},
	\end{align*}
	i.e., formula for Laplace transform of Riemann-Liouville fractional derivative of order $\zeta \in (m-1,m)$.
	
	In particular for $m=1$, if we set the parameters $\zeta_{1}=1$ and $0<\zeta_{0}=\zeta<1$, Equation (\ref{LaplaceDN}) reduces to the Laplace transform of Riemann-Liouville fractional derivative of order $\zeta \in (0,1)$, i.e.,
	\begin{align*}
		\mathcal{L}\{\mathcal {D}^{\varrho_{1}}_{0+,t}g(t)\}=s^{\zeta}	 \mathcal{L}\{g(t)\}-J^{1-\zeta}_{0+,t}g(t)|_{t=0}.
	\end{align*}
\end{remark}
\begin{remark}
	On using $\zeta_{m-1}=...=\zeta_{0}=1$ and $\zeta_{m}=1+\zeta-m$, where $\zeta_{m}\in (0,1)$, in (\ref{LaplaceDN}), we obtain the formula for Laplace transform of Caputo fractional derivative of order in $\zeta \in (m-1,m)$, i.e.,
	\begin{align*}
		\mathcal{L}\{\mathcal {D}^{\varrho_{m}}_{0+,t}g(t)\}=s^{\zeta}\mathcal{L}\{g(t)\}-\sum_{i=0}^{m-1}s^{\zeta-i-1}\frac{d^{i}}{dt^{i}}g(t)|_{t=0}.
	\end{align*}
	In particular, for $m=1$, we have the Laplace transform of Caputo fractional derivative of order $\zeta \in (0,1)$ by setting $0<\zeta_{1}=\zeta<1$ and $\zeta_{0}=1$ in (\ref{LaplaceDN}), i.e.,
	\begin{align*}
		\mathcal{L}\{\mathcal {D}^{\varrho_{1}}_{0+,t}g(t)\}=s^{\zeta}	 \mathcal{L}\{g(t)\}-s^{\zeta-1}g(t)|_{t=0}.
	\end{align*}
\end{remark}

\begin{remark}
	On setting $\zeta_{m-1}=...=\zeta_{1}=1$, $\zeta_{m}=1-\beta(m-\zeta)$,  $\zeta_{0}=1-(m-\zeta)(1-\beta)$, where $0<0<\zeta_{0}, \zeta_{m}<1$, Equation (\ref{LaplaceDN}) yields the formula for Laplace transform of Hilfer fractional derivative of order $\zeta \in (m-1,m)$ and type $\beta\in[0,1]$, i.e.,
	\begin{align*}
		\mathcal{L}\{\mathcal {D}^{\varrho_{m}}_{0+,t}g(t)\}=s^{\zeta}	 \mathcal{L}\{g(t)\}-\sum_{i=0}^{m-1}s^{i-\beta(m-\alpha)}\frac{d^{m-i-1}}{dt^{m-i-1}}J^{(m-\zeta)(1-\beta)}_{0+,t}g(t)|_{t=0},
	\end{align*}
	Similarly, for $m=1$, $\zeta_{1}=1-\beta(1-\zeta)$ and $\zeta_{0}=1-(1-\zeta)(1-\beta)$, where $0<\zeta_{0},\zeta_{1}<1$ in Equation (\ref{LaplaceDN}), we obtain
	\begin{align*}
		\mathcal{L}\{\mathcal {D}^{\varrho_{1}}_{0+,t}g(t)\}=s^{\zeta}	 \mathcal{L}\{g(t)\}-s^{-\beta(1-\zeta)}J^{(1-\zeta)(1-\beta)}_{0+,t}g(t)|_{t=0},
	\end{align*}
	i.e. Laplace transform for Hilfer fractional derivative of order $\zeta \in (0,1)$ and type $\beta\in[0,1]$.
\end{remark}

\begin{lemma}\label{DzhrSamkolemma}
	Let $g_{i}$ be a sequence of functions defined on $(0,b]$ for each $i\in\mathbb{N}$, such that the following conditions hold:
	\begin{enumerate}
		\item Derivatives $D^{\zeta_{0}}_{0+,t}g_{i}(t)$,
		$D^{\zeta_{1}}_{0+,t}D^{\zeta_{0}}_{0+,t}g_{i}(t)$, ...,
		$D^{\zeta_{m-1}}_{0+,t}...D^{\zeta_{0}}_{0+,t}g_{i}(t)$ for $i\in \mathbb{N}, t\in (0,b]$ exist,
		
		%		\item for a given $\varrho_{m}>0$, the Dzherbashian-Nersesian fractional derivatives $\mathcal{D}^{\varrho_{m}}_{0+,t}g_{i}(t)$, for $i\in \mathbb{N}, t\in (0,b]$ exist,
		\item the series $\sum_{i=1}^{\infty}g_{i}(t)$ and  $\sum_{i=1}^{\infty}D^{\zeta_{0}}_{0+,t}g_{i}(t)$, $\sum_{i=1}^{\infty}D^{\zeta_{1}}_{0+,t}D^{\zeta_{0}}_{0+,t}g_{i}(t)$, ...,\\
		$\sum_{i=1}^{\infty}D^{\zeta_{m-1}}_{0+,t}... D^{\zeta_{1}}_{0+,t}D^{\zeta_{0}}_{0+,t}g_{i}(t)$ are
		uniformly convergent on the interval $[a+\varepsilon,b]$ for any $\varepsilon>0$.
	\end{enumerate}
	Then
	\begin{align*}
		\mathcal{D}^{\varrho_{m}}_{0+,t}	 \displaystyle\sum_{i=1}^{\infty}g_{i}(t)=\displaystyle\sum_{i=1}^{\infty}\mathcal{D}^{\varrho_{m}}_{0+,t}g_{i}(t).
	\end{align*}
\end{lemma}
\begin{proof}
	From Equation (\ref{DzhrNer-mathematical}), we have
	\begin{align*}
		 \mathcal{D}^{\varrho_{m}}_{0+,t}\displaystyle\sum_{i=1}^{\infty}g_{i}(t)=\big(J^{1-\zeta_{m}}_{0+,t}D^{\zeta_{m-1}}_{0+,t}D^{\zeta_{m-2}}_{0+,t}...D^{\zeta_{1}}_{0+,t}D^{\zeta_{0}}_{0+,t}\big)\displaystyle\sum_{i=1}^{\infty}g_{i}(t).
	\end{align*}
	Using Lemma (see \cite{samko-book}, page 278, Lemma 15.2), we obtain
	\begin{align*}
		 \mathcal{D}^{\varrho_{m}}_{0+,t}\displaystyle\sum_{i=1}^{\infty}g_{i}(t)&=\big(J^{1-\zeta_{m}}_{0+,t}D^{\zeta_{m-1}}_{0+,t}D^{\zeta_{m-2}}_{0+,t}...D^{\zeta_{1}}_{0+,t}\big)\displaystyle\sum_{i=1}^{\infty}D^{\zeta_{0}}_{0+,t}g_{i}(t).
	\end{align*}
	Continuing in the same manner, we conclusively have
	\begin{align*}
		\mathcal{D}^{\varrho_{m}}_{0+,t}	 \displaystyle\sum_{i=1}^{\infty}g_{i}(t)=\displaystyle\sum_{i=1}^{\infty}\mathcal{D}^{\varrho_{m}}_{0+,t}g_{i}(t).
	\end{align*}
\end{proof}
\subsection{Inverse Source Problems}
In this subsection we are going to investigate two inverse source problems. Firstly, we will study space dependent inverse source problem (\ref{prbmDN})-(\ref{spaceoverdetcdn}). Next, time dependent inverse source problem (\ref{prbmDN})-(\ref{bndryDN}) with integral over-determination condition (\ref{timeoverdetcdn}) is considered. The existence and uniqueness results for both inverse problems are established.

\begin{theorem}\label{spacetheorem1} Let $|\varepsilon|<1$, $0<\varrho_{m}<1$, and
	\begin{itemize}
		\item[(1)] $\varphi_{n}(x)\in C^2\big([0,\pi]\big)$, where $n=0,...,m-1$, be such that $\varphi_{n}(0)=0=\varphi_{n}(\pi)$.
		
		\item[(2)]  $\psi(x)\in C^4\big([0,\pi]\big)$ be such
		that $\psi^{(i)}(0)=0=\psi^{(i)}(\pi)$, for $i=0,2$.
	\end{itemize}
	Then there exists a regular solution of the inverse problem (\ref{prbmDN})-(\ref{spaceoverdetcdn}).
\end{theorem}
\begin{proof}
	In the first step, we construct the solution of inverse problem while in the subsequent steps we prove the existence and uniqueness of solution.
	
	Construction of the Solution:
	
	In consideration of the fact, that the set of eigenfunctions $\{X_{k}:k\in \mathbb{N}\}$ forms an orthonormal basis in $L^{2}([0,\pi])$, the solution of inverse problem (\ref{prbmDN})-(\ref{spaceoverdetcdn}) can be written as
	
\begin{align}
	 u(x,t)&=u_{10}(t)X_{1}(x)+\sum_{k=1}^{\infty}\Big(u_{1k}(t)\;X_{2k+1}(x)+u_{2k}(t)\;X_{2k}(x)\Big),\label{spaceU}\\
	f(x)&=f_{10}X_{1}(x)+\sum_{k=1}^{\infty}\Big(f_{1k} \: X_{2k+1}(x)+f_{2k}\: X_{2k}(x)\Big),\label{spaceF}
\end{align}
where $u_{10}(t),u_{1k}(t), u_{2k}(t), f_{10}, f_{1k}$, and $f_{2k}$ are the unknowns to be determined.

	On using Equations (\ref{spaceU}) and (\ref{spaceF}) in Equation (\ref{prbmDN}), we obtain the following system of FDEs
	\begin{align}
		 \mathcal{D}^{\varrho_{m}}_{0+,t}u_{10}(t)+\lambda_{1}u_{10}(t)&=f_{10},\label{spacefdeU1}\\
		 \mathcal{D}^{\varrho_{m}}_{0+,t}u_{1k}(t)+\lambda_{2k+1}u_{1k}(t)&=f_{1k},\label{spacefdeUk}\\
		 \mathcal{D}^{\varrho_{m}}_{0+,t}u_{2k}(t)+\lambda_{2k}u_{2k}(t)&=f_{2k}.\label{spacefdeVk}
	\end{align}
	By means of Lemma \ref{LaplaceDN}, solutions to Equations (\ref{spacefdeU1})-(\ref{spacefdeVk}) are
	\begin{align}
		u_{10}(t)&=\displaystyle \sum_{n=0}^{m-1}\varphi_{(10)n} e_{\varrho_{m},\varrho_{n}+1}(t;\lambda_{1})+f_{10}\;e_{\varrho_{m},\varrho_{m}+1}(t;\lambda_{1}),\label{FDEU10}\\
		u_{1k}(t)&=\displaystyle \sum_{n=0}^{m-1}\varphi_{(1k)n}\; e_{\varrho_{m},\varrho_{n}+1}(t;\lambda_{2k+1})+f_{1k}\;e_{\varrho_{m},\varrho_{m}+1}(t;\lambda_{2k+1}),\label{FDEU1K}\\
		u_{2k}(t)&=\sum_{n=0}^{m-1}\varphi_{(2k)n}\; e_{\varrho_{m},\varrho_{n}+1}(t;\lambda_{2k})+f_{2k}\;e_{\varrho_{m},\varrho_{m}+1}(t;\lambda_{2k}).\label{FDEU2K}
	\end{align}
	respectively. Furthermore, $\varphi_{(10)n}, \varphi_{(1k)n}$ and $\varphi_{(2k)n}$ are the coefficients of series expansion of $\varphi_{n}(x)$ and are defined as
	$$\varphi_{(10)n}=\langle\varphi_{n}(x),X_{1}(x)\rangle,\;
 \varphi_{(1k)n}=\langle\varphi_{n}(x),X_{2k+1}(x)\rangle, \varphi_{(2k)n}=\langle\varphi_{n}(x),X_{2k}(x)\rangle.$$
Due to over-determination condition (\ref{spaceoverdetcdn}) and from Equations (\ref{FDEU10}), (\ref{FDEU1K}) and (\ref{FDEU2K}), we have
	\begin{align}
		 f_{10}&=\frac{\psi_{10}-\sum_{n=0}^{m-1}\varphi_{(10)n} e_{\varrho_{m},\varrho_{n}+1}(T;\lambda_{1})}{e_{\varrho_{m},\varrho_{m}+1}(T;\lambda_{1})}\label{spaceF10k},\\
		 f_{1k}&=\frac{\psi_{1k}-\sum_{n=0}^{m-1}\varphi_{(1k)n} e_{\varrho_{m},\varrho_{n}+1}(T;\lambda_{2k+1})}{e_{\varrho_{m},\varrho_{m}+1}(T;\lambda_{2k+1})}\label{spaceF1k},\\
		 f_{2k}&=\frac{\psi_{2k}-\sum_{n=0}^{m-1}\varphi_{(2k)n} e_{\varrho_{m},\varrho_{n}+1}(T;\lambda_{2k})}{e_{\varrho_{m},\varrho_{m}+1}(T;\lambda_{2k})}\label{spaceF2k},
	\end{align}	
 where $\psi_{10}, \psi_{1k}$ $\psi_{2k}$ are the coefficients of series expansion of $\psi_(x)$ and are defined as
	\begin{align*}
		\psi_{10}=\langle\psi(x),X_{1}(x)\rangle,\;
		\psi_{1k}=\langle\psi(x),X_{2k+1}(x)\rangle,\;
		\psi_{2k}=\langle\psi(x),X_{2k}(x)\rangle.
	\end{align*}
Plugging (\ref{FDEU10})-(\ref{FDEU2K}) and (\ref{spaceF10k})-(\ref{spaceF2k}) into Equation (\ref{spaceU}) yields
\begin{align}
u(x,t)&=\Bigg[\displaystyle \sum_{n=0}^{m-1}\varphi_{(10)n} e_{\varrho_{m},\varrho_{n}+1}(t;\lambda_{1})\notag\\
&+\Bigg(\frac{\psi_{10}-\sum_{n=0}^{m-1}\varphi_{(10)n} e_{\varrho_{m},\varrho_{n}+1}(T;\lambda_{1})}{e_{\varrho_{m},\varrho_{m}+1}(T;\lambda_{1})}\Bigg)e_{\varrho_{m},\varrho_{m}+1}(t;\lambda_{1})\Bigg]\sqrt{\frac{2}{\pi}}\sin x \notag \\
&+\sum_{k=1}^{\infty}\Bigg[\Bigg(\displaystyle \sum_{n=0}^{m-1}\varphi_{(1k)n}\; e_{\varrho_{m},\varrho_{n}+1}(t;\lambda_{2k+1})\notag \\
&+\Bigg(\frac{\psi_{1k}-\sum_{n=0}^{m-1}\varphi_{(1k)n} e_{\varrho_{m},\varrho_{n}+1}(T;\lambda_{2k+1})}{e_{\varrho_{m},\varrho_{m}+1}(T;\lambda_{2k+1})}\Bigg)\;e_{\varrho_{m},\varrho_{m}+1}(t;\lambda_{2k+1}) \Bigg)\notag \\
&\sqrt{\frac{2}{\pi}}\sin(2k+1)x+\Bigg(\sum_{n=0}^{m-1}\varphi_{(2k)n}\; e_{\varrho_{m},\varrho_{n}+1}(t;\lambda_{2k}) \notag\\
&+\Bigg(\frac{\psi_{2k}-\sum_{n=0}^{m-1}\varphi_{(2k)n} e_{\varrho_{m},\varrho_{n}+1}(T;\lambda_{2k})}{e_{\varrho_{m},\varrho_{m}+1}(T;\lambda_{2k})}\Bigg)e_{\varrho_{m},\varrho_{m}+1}(t;\lambda_{2k})\Bigg) \notag\\
&\sqrt{\frac{2}{\pi}}\label{spacesolution}\sin 2kx\Bigg].
\end{align}	
	Existence of the Solution: We intend to demonstrate that the series solution obtained is indeed a regular solution. We show this by the virtue of Weierstrass M-test, that the series representations of solution of inverse problem are uniformly convergent.
	
	We first show that the series representation of  $f(x)$ is uniformly convergent. By using Lemmata \ref{podlem} and \ref{AliMittagtypeLemma1} in  Equation (\ref{spaceF10k}), we obtain
	\begin{align*}
		|f_{10}|\leq C_{2}\lambda_{1}\Big(|\psi_{10}|+\sum_{n=0}^{m-1}\frac{C_{1}T^{\varrho_{n}-\varrho_{m}}}{\lambda_{1}}|\varphi_{(10)n}|\Big).
	\end{align*}
	By Cauchy-Schwartz inequality
	\begin{align}
		|f_{10}|\leq C_{2}(1-\varepsilon)\|\psi\|_{x}+M_{1}\displaystyle\sum_{n=0}^{m-1}T^{\varrho_{n}-\varrho_{m}}\|\varphi_{n}\|_{x}.\label{spaceF10kestimate}
	\end{align}
	\noindent where $M_{1}=C_{1}C_{2}$.
	
	Similarly, using Lemmata \ref{podlem} and \ref{AliMittagtypeLemma1} along with Lemmata \ref{estimatelemma1} and \ref{estimatelemma2} in Equations (\ref{spaceF1k})-(\ref{spaceF2k}) yield
	\begin{align}
		|f_{1k}|&\leq \frac{9C_{2}(1-\varepsilon)}{k^{2}}\|\psi^{(iv)}\|_{x}+\frac{M_{1}}{k^{2}}\displaystyle\sum_{n=0}^{m-1}T^{\varrho_{n}-\varrho_{m}}\|\varphi''_{n}\|_{x},\label{spaceF1kestimate}\\
		|f_{2k}|&\leq \frac{4C_{2}(1+\varepsilon)}{k^{2}}\|\psi^{(iv)}\|_{x}+\frac{M_{1}}{k^{2}}\sum_{n=0}^{m-1}T^{\varrho_{n}-\varrho_{m}}\|\varphi''_{n}\|_{x}.\label{spaceF2kestimate}
	\end{align}
	With the aid of estimates (\ref{spaceF10kestimate})-(\ref{spaceF2kestimate}) and Equation (\ref{spaceF}), evidently $f(x)$ is bounded above by convergent series. Hence, by Weierstrass M-test, it represents a continuous function.
	
Using Lemma \ref{podlem} and inequality (\ref{lamdasestimate}) in Equation (\ref{FDEU10}), we obtain
	\begin{align*}
		|u_{10}(t)|&\leq \frac{C_{1}}{1-\varepsilon}\Big(\sum_{n=0}^{m-1}t^{\varrho_{n}-\varrho_{m}}|\varphi_{(10)n}|+|f_{10}|\Big). 	 
	\end{align*}
	Using inequality (\ref{spaceF10kestimate}) and Cauchy-Schwarz inequality, we have
	\begin{align}
		t^{\varrho_{m}}|u_{10}(t)|&\leq \frac{C_{1}}{1-\varepsilon}\Big(\sum_{n=0}^{m-1}t^{\varrho_{n}}\|\varphi_{n}\|_{x}+t^{\varrho_{m}}|f_{10}|\Big). \label{U10estimate}
%		|u_{10}(t)|&\leq \frac{C_{1}}{1-\varepsilon}\Big(\sum_{n=0}^{m-1}(t^{\varrho_{n}-\varrho_{m}}+M_{1}T^{\varrho_{n}-\varrho_{m}})\|\varphi_{n}\|_{2}\Big)+M_{1}\|\psi\|_{2}.\label{U10estimate}
	\end{align}
	Similarly,
\begin{align}
	t^{\varrho_{m}}|u_{1k}(t)|\leq \frac{C_{1}}{(1-\varepsilon)k^{2}}\Big(\sum_{n=0}^{m-1}t^{\varrho_{n}}\|\varphi_{n}\|_{x}+t^{\varrho_{m}}|f_{1k}|\Big),\label{U1kestimate}\\
	t^{\varrho_{m}}|u_{2k}(t)|\leq \frac{C_{1}}{(1+\varepsilon)k^{2}}\Big(\sum_{n=0}^{m-1}t^{\varrho_{n}}\|\varphi_{n}\|_{x}+t^{\varrho_{m}}|f_{2k}|\Big).\label{U2kestimate}
\end{align}	
	
%	\begin{align}
%		|u_{1k}(t)|&\leq \frac{C_{1}}{(1-\varepsilon)k^{2}}\sum_{n=0}^{m-1}(t^{\varrho_{n}-\varrho_{m}}+M_{1}T^{\varrho_{n}-\varrho_{m}})\|\varphi_{n}\|_{2}+M_{1}\frac{\|\psi''\|_{2}}{k^{2}},\label{U1kestimate}\\	 
%		|u_{2k}(t)|&\leq \frac{C_{1}}{(1+\varepsilon)k^{2}}\sum_{n=0}^{m-1}(t^{\varrho_{n}-\varrho_{m}}+M_{1}T^{\varrho_{n}-\varrho_{m}})\|\varphi_{n}\|_{2}+M_{1}\frac{\|\psi''\|_{2}}{k^{2}}.\label{U2kestimate}	 
%	\end{align}
	Making use of estimates (\ref{spaceF10kestimate})-(\ref{U2kestimate}) in Equation (\ref{spaceU}),
	it is evident that $t^{\varrho_{m}}u(x,t)$ is bounded above by convergent series. Therefore, by Weierstrass M-test, $t^{\varrho_{m}}u(x,t)$ represents a continuous function.
	
	Next we prove the convergence of $t^{\varrho_{m}}\mathcal{D}^{\varrho_{m}}_{0+,t}u(x,t)$, $t^{\varrho_{m}}u_{xx}(x,t)$ and $t^{\varrho_{m}}u_{xx}(\pi-x,t)$. Firstly, we prove the uniform convergence of the series representation of $t^{\varrho_{m}}\mathcal{D}^{\varrho_{m}}_{0+,t}u(x,t)$ with the aid of Lemma \ref{DzhrSamkolemma}. For this we need to show the uniform convergence of $\{ \sum_{k=1}^{\infty}u_{ik}(t): i=1,2\}$ and $\{ \sum_{k=1}^{\infty}\mathcal{D}^{\varrho_{m}}_{0+,t}u_{ik}(t): i=1,2\}$. Since, $\{ \sum_{k=1}^{\infty}u_{ik}(t): i=1,2\}$ is convergent by using estimates (\ref{U1kestimate})-(\ref{U2kestimate}).
	Therefore, we need only to show that $\{ \sum_{k=1}^{\infty}\mathcal{D}^{\varrho_{m}}_{0+,t}u_{ik}(t): i=1,2\}$ is uniformly convergent.
	
	On using Equation (\ref{spacefdeU1}), we can write
	\begin{align*}
		|\mathcal{D}^{\varrho_{m}}_{0+,t}u_{10}(t)|\leq (1-\varepsilon)|u_{10}(t)|+|f_{10}|.
	\end{align*}
	Using estimate (\ref{U10estimate}) and Cauchy-Schwarz inequality,
	\begin{align}
		 t^{\varrho_{m}}|\mathcal{D}^{\varrho_{m}}_{0+,t}u_{10}(t)|&\leq C_{1}\sum_{n=0}^{m-1}t^{\varrho_{n}}\|\varphi_{n}\|_{x}+(1+C_{1})t^{\varrho_{m}}|f_{10}|.\label{spacedu10est}
%		|\mathcal{D}^{\varrho_{m}}_{0+,t}u_{10}(t)|&\leq
%		 C_{1}\Big(\sum_{n=0}^{m-1}(t^{\varrho_{n}-\varrho_{m}}+M_{1}T^{\varrho_{n}-\varrho_{m}})\|\varphi_{n}\|_{2}\Big)\notag \\
%		&+M_{1}(1-\varepsilon)\|\psi\|_{2}+|f_{10}|.
	\end{align}
	Similarly,
	\begin{align}
		 t^{\varrho_{m}}\big|\displaystyle\sum_{k=1}^{\infty}\mathcal{D}^{\varrho_{m}}_{0+,t}u_{1k}(t)\big| &\leq \sum_{k=1}^{\infty} \Big( C_{1}\sum_{n=0}^{m-1}\frac{t^{\varrho_{n}}}{k^{2}}\|\varphi''_{n}\|_{x}+(1+C_{1})t^{\varrho_{m}}|f_{1k}|\Big),\label{spacedu1kest}\\
		 t^{\varrho_{m}}\big|\displaystyle\sum_{k=1}^{\infty}\mathcal{D}^{\varrho_{m}}_{0+,t}u_{2k}(t)\big| &\leq \sum_{k=1}^{\infty} \Big(C_{1}\sum_{n=0}^{m-1}\frac{t^{\varrho_{n}}}{k^{2}}\|\varphi''_{n}\|_{x}+(1+C_{1})t^{\varrho_{m}}|f_{2k}|\Big).\label{spacedu2kest}
%		 \big|\displaystyle\sum_{k=1}^{\infty}\mathcal{D}^{\varrho_{m}}_{0+,t}u_{1k}(t)\big| &\leq\sum_{k=1}^{\infty}\frac{1}{k^{2}}\Big(C_{1}\sum_{n=0}^{m-1}(t^{\varrho_{n}-\varrho_{m}}+M_{1}T^{\varrho_{n}-\varrho_{m}})\|\varphi''_{n}\|_{2}\notag\\+9M_{1}(1-\varepsilon)
%		 &\|\psi^{(iv)}\|_{2}\Big)+\sum_{k=1}^{\infty}|f_{1k}|,\label{spacedu1kest}\\
%		 \big|\displaystyle\sum_{k=1}^{\infty}\mathcal{D}^{\varrho_{m}}_{0+,t}u_{2k}(t)\big| &\leq\sum_{k=1}^{\infty}\frac{1}{k^{2}}\Big(C_{1}\sum_{n=0}^{m-1}(t^{\varrho_{n}-\varrho_{m}}+M_{1}T^{\varrho_{n}-\varrho_{m}})\|\varphi''\|_{2}\notag\\+4M_{1}(1+\varepsilon)
%		 &\|\psi^{(iv)}\|_{2}\Big)+\sum_{k=1}^{\infty}|f_{2k}|.\label{spacedu2kest}
	\end{align}
	Taking into account the Lemma \ref{DzhrSamkolemma} i.e. $\mathcal{D}^{\varrho_{m}}_{0+,t}	 \sum_{i=1}^{\infty}h_{i}(t)=\sum_{i=1}^{\infty}\mathcal{D}^{\varrho_{m}}_{0+,t}h_{i}(t)$, and estimates (\ref{spaceF10kestimate})-(\ref{spaceF2kestimate}) and (\ref{spacedu10est})-(\ref{spacedu2kest}), we see that $\mathcal{D}^{\varrho_{m}}_{0+,t}u_(x,t)$ is bounded above by convergent numerical series. Hence, $t^{\varrho_{m}}\mathcal{D}^{\varrho_{m}}_{0+,t}u_(x,t)$ represents a continuous function.
	
	Now we prove the convergence of space derivatives $t^{\varrho_{m}}u_{xx}(x,t)$ and $t^{\varrho_{m}}u_{xx}(\pi-x,t)$. It suffices to prove the convergence of series representation of  $t^{\varrho_{m}}u_{xx}(x,t)$ due to the fact that $|t^{\varrho_{m}}u_{xx}(x,t)|=|t^{\varrho_{m}}u_{xx}(\pi-x,t)|$.
	\begin{align*}
		|u_{xx}(x,t)|\leq |u_{10}(t)|+\sum_{k=1}^{\infty}\Big((2k+1)^{2}|u_{1k}(t)|+4k^{2}|u_{2k}(t)|\Big).
	\end{align*}
	By means of Lemmata \ref{estimatelemma1} and \ref{estimatelemma2} alongside estimates (\ref{U1kestimate}) and (\ref{U2kestimate}), we have
	\begin{align*}
		t^{\varrho_{m}}|u_{xx}(x,t)|&\leq \frac{C_{1}}{1-\varepsilon}\Big(\sum_{n=0}^{m-1}t^{\varrho_{n}}\|\varphi_{n}\|_{x}+t^{\varrho_{m}}|f_{10}|\Big)\\
		 &+\sum_{k=1}^{\infty}\frac{9C_{1}}{1-\varepsilon}\Big(\sum_{n=0}^{m-1}\frac{t^{\varrho_{n}}}{k^{2}}\|\varphi''_{n}\|_{x}+t^{\varrho_{m}}|f_{1k}|\Big)\\
		 &+\sum_{k=1}^{\infty}\frac{4C_{1}}{1+\varepsilon}\Big(\sum_{n=0}^{m-1}\frac{t^{\varrho_{n}}}{k^{2}}\|\varphi''_{n}\|_{x}+t^{\varrho_{m}}|f_{2k}|\Big).
	\end{align*}
	It is clear that $u_{xx}(x,t)$ is bounded above by convergent numerical series. Accordingly, by the virtue of Weierstrass M-test, $u_{xx}(x,t)$ represents a continuous function.
	\end{proof}
Uniqueness of the Solution: Let $\{u_1(x,t), f_1(x)\}$ and $\{u_2(x,t), f_2(x)\}$ be two regular solution sets of the
inverse source problem (\ref{prbmDN})- (\ref{spaceoverdetcdn}). Then
$u_1(x_0,t)=u_2(x_0,t)$ for some point $x_0\in (0,\pi)$ implies $f_1(x)=f_2(x)$. The proof can be obtained by following the strategy of Theorem 3.2 of \cite{AliFCAA}.
\begin{remark} For $m=1$, $\xi_{0}=1$ and $\xi_{1}=\zeta$, solution of inverse problem (\ref{prbmDN})-(\ref{spaceoverdetcdn}) that was obtained in \ref{spacesolution} becomes
\begin{align*}
u(x,t)&=\Big(\varphi_{10}e_{\zeta,1}(t;\lambda_{1})
+\Big(\frac{\psi_{10}-\varphi_{10}e_{\zeta,1}(T;\lambda_{1})}{e_{\zeta,\zeta+1}(T;\lambda_{1})}\Big)e_{\zeta,\zeta+1}(t;\lambda_{1})\Big)\sqrt{\frac{2}{\pi}}\sin x \\
&+\sum_{k=1}^{\infty}\Big[\Big(\varphi_{1k}e_{\zeta,1}(t;\lambda_{2k+1})\\
&+\Big(\frac{\psi_{1k}-\varphi_{1k}e_{\zeta,1}(T;\lambda_{2k+1})}{e_{\zeta,\zeta+1}(T;\lambda_{2k+1})}\Big)e_{\zeta,\zeta+1}(t;\lambda_{2k+1})\Big)\sqrt{\frac{2}{\pi}}\sin(2k+1)x\\
&+\Big(\varphi_{2k}e_{\zeta,1}(t;\lambda_{2k})\\
&+\Big(\frac{\psi_{2k}-\varphi_{2k}e_{\zeta,1}(T;\lambda_{2k})}{e_{\zeta,\zeta+1}(T;\lambda_{2k})}\Big)e_{\zeta,\zeta+1}(t;\lambda_{2k})\Big)\sqrt{\frac{2}{\pi}}\sin 2kx \Big],
\end{align*}
which is exactly the same that was obtained in \cite{Tore-Tapd}.
%However, the boundary condition $g'(0)=0$ does not arise in our case. The reason is that substituion of incorrect value of $\sin n\pi$, where $n\in \mathbb{N}$, resulted in extra boundary condition in \cite{Tore-Tapd}.
\end{remark}
Next, we are concerned with problem (\ref{prbmDN})-(\ref{bndryDN}) alongside the over-specified condition (\ref{timeoverdetcdn}). We shall determine pair of functions $\{u(x,t),a(t)\}$, whereas $f(x,t)$ is known.
\begin{theorem}\label{td1theorem}
	Let $|\varepsilon|<1$, $0<\varrho_{m}<1$ and $\varphi$, $f(.,t)$ $\in$ $C^{2}([0,\pi])$ be such that $\varphi_{n}(0)=0=\varphi_{n}(\pi)$, where $n=0,...,m-1$, $f(0,t)=0=f(\pi,t)$, and there exists $M_{2}$ satisfying $0<\frac{1}{M_{2}}\leq |\int_{0}^{\pi} f(x,t)dx|$. Then there exists a regular solution of the inverse problem (\ref{prbmDN})-(\ref{bndryDN}) together with the over-determination condition (\ref{timeoverdetcdn}).
\end{theorem}
\begin{proof}
As in the proof of Theorem \ref{spacetheorem1}, we shall first construct the solution of the time dependent inverse source problem.

Construction of Solution: Expanding $u(x,t)$ and $f(x)$ by means of orthogonal functions, we have
	\begin{align}
		 u(x,t)&=u_{10}(t)X_{1}(x)+\sum_{k=1}^{\infty}\Big(u_{1k}(t)X_{2k+1}(x)+u_{2k}(t)X_{2k}(x)\Big),\label{timeU}\\
		 f(x,t)&=f_{10}(t)X_{1}(x)+\sum_{k=1}^{\infty}\Big(f_{1k}(t)X_{2k+1}(x)+f_{2k}(t)X_{2k}(x)\Big).\label{timeF}
	\end{align}
	In accordance with the same procedure as in proof of Theorem \ref{spacetheorem1}, we have
	\begin{align}\label{timesolnprefdessol}
		u(x,t)&=\Big(\sum_{n=0}^{m-1}\varphi_{(10)n} e_{\varrho_{m},\varrho_{n}+1}(t;\lambda_{1})+e_{\varrho_{m},\varrho_{m}}(t;\lambda_{1})*a(t)f_{10}(t)\Big)X_{1}(x)\notag\\
		&+\sum_{k=1}^{\infty}\Big(\sum_{n=0}^{m-1}\varphi_{(1k)n} e_{\varrho_{m},\varrho_{n}+1}(t;\lambda_{2k+1})+e_{\varrho_{m},\varrho_{m}}(t;\lambda_{2k+1})\notag\\
		 &*a(t)f_{1k}(t)\Big)X_{2k+1}(x)+\sum_{k=1}^{\infty}\Big(\sum_{n=0}^{m-1}\varphi_{(2k)n} e_{\varrho_{m},\varrho_{n}+1}(t;\lambda_{2k})\notag\\
		 &+e_{\varrho_{m},\varrho_{m}}(t;\lambda_{2k})*a(t)f_{2k}(t)\Big)X_{2k}(x).
	\end{align}
	Applying the fractional Dzherbashian-Nersesian operator on over-determination condition (\ref{timeoverdetcdn}) and using equation (\ref{prbmDN}), we have
	\begin{align*}\label{operatoreq}
		a(t)=\Big(\int_{0}^{\pi}f(x,t)dx \Big)^{-1}\Big(\mathcal{D}^{\varrho_{m}}_{0+,t}E(t)+\mathcal{F}(t)+\int_{0}^{t}\mathcal{K}(t,\tau)a(\tau)d\tau\Big),
	\end{align*}
	where
	\begin{eqnarray}\label{Fourier}
		 \mathcal{F}(t)&=2(1+\varepsilon)\sqrt{\frac{2}{\pi}}\Bigg(\displaystyle\sum_{n=0}^{m-1}\varphi_{(10)n}e_{\varrho_{m},\varrho_{n}+1}(t,\lambda_{1})+\displaystyle\sum_{k=1}^{\infty}(2k+1)\notag\\
		 &\Big(\displaystyle\sum_{n=0}^{m-1}\varphi_{(1k)n}e_{\varrho_{m},\varrho_{n}+1}(t,\lambda_{2k+1})\Big)\Bigg),
	\end{eqnarray}
	and
	\begin{eqnarray}\label{kappa}
		 \mathcal{K}(t,\tau)&=2(1+\varepsilon)\sqrt{\frac{2}{\pi}}\Big(e_{\varrho_{m},\varrho_{m}}(t-\tau,\lambda_{1})f_{10}(\tau)\notag\\
		 &+\displaystyle\sum_{k=1}^{\infty}(2k+1)e_{\varrho_{m},\varrho_{m}}(t-\tau,\lambda_{1})f_{1k}(\tau)\Big).
	\end{eqnarray}
	Define an operator $\mathcal{B}(a(t)):=a(t)$.
	We are determined to prove that the mapping $\mathcal{B}:C([0,T]) \rightarrow C([0,T])$ is well defined and a contraction mapping. However, firstly we shall prove that the series involved in (\ref{Fourier}) and (\ref{kappa}) are uniformly convergent i.e. for $a(t)\in C([0,T])$, we have $\mathcal{B}(a(t))\in C([0,T])$. Taking into account the Lemmata \ref{podlem} and \ref{estimatelemma1}, we obtain
	\begin{align*}
		t^{\varrho_{m}}|\mathcal{F}(t)|&\leq 2\sqrt{\frac{2}{\pi}}\Big(\frac{1+\varepsilon}{1-\varepsilon}\Big)C_{1}\sum_{n=0}^{m-1}t^{\varrho_{n}}\Big(\|\varphi_{n}\|_{x}+\sum_{k=1}^{\infty}\frac{3}{k^{3}}\|\varphi''_{n}\|_{x}\Big),\\
		(t-\tau)|\mathcal{K}(t,\tau)|&\leq 2\sqrt{\frac{2}{\pi}}\Big(\frac{1+\varepsilon}{1-\varepsilon}\Big)C_{1}\sum_{n=0}^{m-1}\Big(\|f(x,t)\|_{x,t}+\sum_{k=1}^{\infty}\frac{3}{k^{3}}\Big\|f_{xx}(x,t)\|_{x,t}\Big).
	\end{align*}
	Hence by Weierstrass M-Test, the series in (\ref{Fourier}) and (\ref{kappa}) are convergent. Thus, for $a(t)\in C([0,T])$, we have $\mathcal{B}(a(t))\in C([0,T])$. In addition, we may write
	\begin{equation}\label{kappaMtest}
		\|\mathcal{K}(t,\tau)\|_{t \times t}\leq K,\quad t\in (0,T],
	\end{equation}
	where $K$ is some positive constant. Furthermore, we prove that the mapping $\mathcal{B}:C([0,T]) \rightarrow C([0,T])$ is contraction. Consider
	\begin{align*}
		|\mathcal{B}(a)-\mathcal{B}(b)|	\leq \Big(\int_{0}^{\pi}f(x,t)dx \Big)^{-1} \int_{0}^{t}|a(\tau)-b(\tau)|\;|\mathcal{K}(t,\tau)|d\tau,
	\end{align*}
	from inequality (\ref{kappaMtest}), we obtain
	\begin{align*}
		|\mathcal{B}(a)-\mathcal{B}(b)|	\leq TKM_{2}\max\limits_{0 \leq t \leq T}|a(\tau)-b(\tau)|.
	\end{align*}
	Accordingly, we have
	\begin{align*}
		\|\mathcal{B}(a)-\mathcal{B}(b)\|_{t}	\leq TKM_{2}\|a-b\|_{t}.
	\end{align*}
	Consequently, by Banach fixed-point theorem, the mapping $\mathcal{B}(.)$ is a contraction for $T<1/KM_{2}$, which promises the unique determination of $a(.)\in C([0,T])$.
	
Existence of the Solution: To show the existence of the solution of the time dependent inverse source problem, we will show that the series expansion of $t^{\varrho_{m}}u(x,t)$ given by equation (\ref{timesolnprefdessol}), $t^{\varrho_{m}}\mathcal{D}^{\varrho_{m}}_{0+,t}u(x,t)$, $t^{\varrho_{m}}u_{xx}(x,t)$, and $t^{\varrho_{m}}u_{xx}(\pi-x,t)$ are uniformly convergent. As $a(t)\in C([0,T])$, there exists a positive integer $M_{3}$, such that $|a(t)|\leq M_{3}$. Moreover, by using the same procedure as in the proof of Theorem \ref{spacetheorem1}, we obtain the estimates for $u(x,t)$, $\mathcal{D}^{\varrho_{m}}_{0+,t}u(x,t)$ and $u_{xx}(x,t)$:
	\begin{align*}
		t^{\varrho_{m}}|u(x,t)|&\leq  \frac{C_{1}}{1-\varepsilon}\Big(\sum_{n=0}^{m-1}t^{\varrho_{n}}\|\varphi_{n}\|_{x}+M_{3}t^{\varrho_{m}}\|f(x,t)\|_{x,t}\Big)\notag \\
		 &+\frac{C_{1}}{1-\varepsilon}\sum_{k=1}^{\infty}\frac{1}{k^{2}}\Big(\sum_{n=0}^{m-1}t^{\varrho_{n}}\|\varphi''_{n}\|_{x}+M_{3}t^{\varrho_{m}}\|f(x,t)\|_{x,t}\Big)\notag \\ &+\frac{C_{1}}{1+\varepsilon}\sum_{k=1}^{\infty}\frac{1}{k^{2}}\Big(\sum_{n=0}^{m-1}t^{\varrho_{n}}\|\varphi''_{n}\|_{x}+M_{3}t^{\varrho_{m}}\|f(x,t)\|_{x,t}\Big).
	\end{align*}
\begin{align*}
		t^{\varrho_{m}}|\mathcal{D}^{\varrho_{m}}_{0+,t}u(x,t)|&\leq C_{1}\sum_{n=0}^{m-1}t^{\varrho_{n}}\|\varphi_{n}\|_{x}+M_{3}(1+C_{1})t^{\varrho_{m}}\|f(x,t)\|_{x,t}\notag \\
		+& \sum_{k=1}^{\infty}\frac{2C_{1}}{k^{2}}\Big( \sum_{n=0}^{m-1}t^{\varrho_{n}}\|\varphi''_{n}\|_{x}+M_{3}(1+C_{1})t^{\varrho_{m}}\|f_{xx}(x,t)\|_{x,t}\Big).
	\end{align*}	
    \begin{align*}
		t^{\varrho_{m}}|u_{xx}(x,t)|&\leq \frac{C_{1}}{1-\varepsilon}\Big(\sum_{n=0}^{m-1}t^{\varrho_{n}}\|\varphi_{n}\|_{x}+M_{3}t^{\varrho_{m}}\|f(x,t)\|_{x,t}\Big)\\
		 &+\sum_{k=1}^{\infty}\frac{9}{k^{2}}\Big(\frac{C_{1}}{1-\varepsilon}\sum_{n=0}^{m-1}t^{\varrho_{n}}\|\varphi''_{n}\|_{x}+M_{3}t^{\varrho_{m}}\Big\|f_{xx}(x,t)\|_{x,t}\Big)\\
		 &+\sum_{k=1}^{\infty}\frac{4}{k^{x}}\Big(\frac{C_{1}}{1+\varepsilon}\sum_{n=0}^{m-1}t^{\varrho_{n}}\|\varphi''_{n}\|_{x}+M_{3}t^{\varrho_{m}}\|f_{xx}(x,t)\|_{x,t}).
	\end{align*}
	which on using Weierstrass M-test converge.
\end{proof}
Uniqueness of the Solution: The uniqueness of the source term $a(t)$ has already been proved by using the Banach fixed point theorem. Uniqueness of $u(x,t)$ can be obtained by following the similar lines as in Theorem 2.5 of \cite{Sara4th}.
\section{Conclusions}\label{conclusions}
We investigated in this paper that fractional Dzherbashian-Nersesian operator is generalization to Riemann-Liouville, Caputo and Hilfer fractional derivatives. Moreover, the emphasis was on the inverse problems of finding a space dependent source term and a time-dependent source term of the diffusion equation from final temperature distribution and integral over-determination datum respectively. Under some consistency and regularity conditions on the given datum, the existence and uniqueness of inverse problems are presented by using the Fourier Method alongside Weierstrass M-test and Banach fixed point theorem. Finally, under certain fixation of parameters in fractional Dzherbashian-Nersesian operator, our results generalize results of Torebek et al. \cite{Tore-Tapd}.
  %%%%%%%%%%%%%%%%%%%%%%%%%%%%%%%%

%%%%%%%%%% put authors' addresses here, in \it %%%%%%%%

 \bigskip \smallskip

 \it

 \noindent
   %(First) Author's full postal address
   $^{1,3}$ Department of Mathematics\\
   COMSATS University Islamabad, Park Road\\
   Islamabad, PAKISTAN\\[4pt]
   e-mail: anwar.maths@must.edu.pk, anwaarqau@gmail.com (Anwar Ahmad)\\
   e-mail: salman-amin@comsats.edu.pk (Salman A. Malik, Corr.
   author)\\
   \hfill Received: April 04, 2021 \\[12pt]
   % Second Author's address
   $^2$ Department of Sciences and Humanities\\
   National University of Computer and Emerging Sciences\\
   Islamabad, PAKISTAN \\[4pt]
   e-mail: muhammad.ali.pk.84@gmail.com (Muhammad Ali)

\end{document}